\documentclass[11pt]{amsart}

\usepackage[all,cmtip]{xy}
\usepackage{amsmath, amssymb}
\usepackage{mathtools}
\usepackage[T1]{fontenc}
\usepackage[english]{babel}
\usepackage[hidelinks]{hyperref}
\usepackage[marginpar=2.5cm]{geometry}
\usepackage{xcolor}\usepackage{marginnote}
\usepackage{amsrefs}
\usepackage{enumitem}
\usepackage{tikz-cd}

\newtheorem{thm}{Theorem}[section]

\newtheorem*{thm*}{Theorem}
\newtheorem{lem}[thm]{Lemma}

\newtheorem{prop}[thm]{Proposition}
\newtheorem*{prop*}{Proposition}

\newtheorem{cor}[thm]{Corollary}

\theoremstyle{definition}
\newtheorem{defn}[thm]{Definition}
\newtheorem{notation}[thm]{Notation}
\newtheorem{remark}[thm]{Remark}
\newtheorem{question}[thm]{Question}

\newtheorem{example}[thm]{Example}

\DeclareMathOperator{\Span}{span}

\DeclareMathOperator{\rk}{rk}

\DeclareMathOperator{\AG}{AG}

\newcommand{\bb}{\mathbb}

\newcommand{\cc}{\mathcal}

\newcommand{\R}{\mathbb{R}}

\newcommand{\ip}[2]{\left\langle #1, #2 \right\rangle}

\newcommand{\At}{\mathcal{A}}
\newcommand{\zero}{\hat{0}}
\newcommand{\one}{{1}}
\newcommand{\rad}{\mathrm{rad}}

\newcommand{\PG}{\mathrm{PG}}

\newcommand{\BB}{\mathrm{B}}
\newcommand{\de}{\delta}

\title[An Operator-Theory Construction on Geometric Lattices]{An Operator-Theory Construction\\ on Geometric Lattices}
\author[T.\ Sinclair]{Thomas Sinclair}
\address{Department of Mathematics, Purdue University\\ 150 N.\ University St.\\ West Lafayette, IN 47907}
\email{tsincla@purdue.edu}
\thanks{The author was partially supported by NSF grant DMS-2055155.}

\begin{document}

\begin{abstract}
    We introduce a canonical operator-theoretic construction associated to a finite geometric lattice, in which a simple nonassociative ``diamond product'' on the lattice basis gives rise to a family of creation operators indexed by atoms and a corresponding self-adjoint Hamiltonian on $\R[L]$.  A key structural feature is that the Hamiltonian changes rank by at most one, so that its compression to the rank-radial subspace is a Jacobi matrix.  In this way, geometric lattices give rise in a direct and uniform manner to finite orthogonal polynomial systems.

    The Jacobi coefficients admit explicit combinatorial formulas.  For Boolean lattices one obtains the centered Krawtchouk Jacobi matrix, while for projective geometries one obtains natural $q$-deformations consistent with the $q$-Hahn family.  The construction applies to arbitrary geometric lattices and requires no symmetry assumptions.
\end{abstract}

\maketitle


\section{Introduction}

We introduce a canonical operator-theoretic construction associated to a finite geometric lattice, producing a self-adjoint operator where the ``rank-radial'' compression yields a Jacobi matrix.  In this way, a geometric lattice gives rise to a finite orthogonal polynomial system.  For Boolean lattices, the resulting Jacobi matrix is precisely the centered version of the Krawtchouk Jacobi matrix.  For projective geometries, one obtains natural $q$-deformations consistent with the $q$-Hahn family.  In both cases, the coefficients arise from explicit combinatorial data of the lattice, and the construction applies without symmetry assumptions, providing a complementary perspective to the association scheme approach of Delsarte \cites{Delsarte, DelsarteGoethals,BannaiIto}.

The starting point is an elementary bilinear product on $\R[L]$.  If $x,y\in L$, define
\[
x\diamond y=
\begin{cases}
x\vee y,& x\wedge y=\zero,\\
0,& x\wedge y\neq \zero.
\end{cases}
\]
That is, one joins elements exactly when they are disjoint. While the product is commutative and unital, it is generally nonassociative, reflecting the geometry of independence. It is precisely for this reason that it becomes more natural to treat this product as inducing a collection of multiplication operators on the vector space $\R[L]$. 

For each atom $a$, left multiplication defines a creation operator
\[
L_a(x)=a\diamond x,
\]
and taking the self-adjoint part produces a finite-dimensional creation--annihilation system.  Summing over atoms yields the lattice Hamiltonian
\[
H:=\sum_{a\in \At(L)} \frac{L_a+L_a^t}{2},
\]
which serves as the motivating object of the paper.

The operator $H$ appears to capture a surprisingly rich section of the combinatorial data of the geometric lattice. At the surface level, one can already see that it respects rank because of submodularity. Consequently, the Hamiltonian changes rank by at most one, which forces its radial compression to be a tridiagonal (Jacobi) matrix whose coefficients admit an explicit combinatorial formula, providing a canonical and computable bridge from geometric lattices to Jacobi matrices and their associated orthogonal polynomials.

The use of raising and lowering operators in graded posets has precedents, notably in Stanley's theory of differential posets \cite{Stanley-poset}.  The present construction differs in that it is driven by the geometry of independence and produces a self-adjoint Hamiltonian whose radial part is governed by orthogonal polynomial recurrences.  The formal resemblance with creation–annihilation operators also suggests connections with free probability and interacting Fock spaces. We refer the reader to \cite{NicaSpeicher} for general background on free probability and combinatorics and to \cite{Anshelevich} for connections to Jacobi matrices and orthogonal polynomials specifically.

\section{The Diamond Product}

\begin{notation}
    Throughout, $L$ denotes a finite geometric lattice of rank $r=\rk(\one)$.  We write $\At(L)$ for the set of atoms of $L$. 
    To avoid confusion with the vector $0$, we will use $\hat 0$ to denote the bottom element of the lattice.
\end{notation}

We refer to standard sources such as \cite{Oxley} for background on geometric lattices and matroids.

\begin{defn}
Let $\R[L]$ be the real vector space with basis $\{e_x:x\in L\}$ indexed by the lattice elements.
We identify basis vectors with lattice elements when convenient.  Define a bilinear product on
$\R[L]$ by
\[
x\diamond y=
\begin{cases}
x\vee y,& x\wedge y=\zero,\\
0,& x\wedge y\neq \zero.
\end{cases}
\]
We call this the \emph{diamond product}.
\end{defn}

\begin{remark}
The product is commutative by symmetry of meet and join.  Since $\zero\wedge x=\zero$ and $\zero\vee x=x$ for all $x$, the lattice bottom $\zero$ is a multiplicative unit:
\[
    \zero\diamond x=x\diamond \zero=x.
\]
\end{remark}

\begin{example}
If $L$ is Boolean, then $\diamond$ is simply disjoint union on subsets.  In that case the product is associative and the resulting algebra identifies with the squarefree algebra
\[
    \R[x_1,\dots,x_n]/(x_1^2,\dots,x_n^2),
\]
after choosing the atoms as generators.  
\end{example}

The previous example can be cast as a consequence of Boolean lattices possessing a ``metric'' version of modularity given by additivity of the rank function. We refer the reader to \cite{CMS} for background.

\begin{prop}
    If the rank function on $L$ satisfies
    \[
    \rk(x) + \rk(y) = \rk(x\vee y) + \rk(x\wedge y)
    \]
    for all $x,y\in L$, then the diamond product is associative.
\end{prop}

\begin{proof}
Let $x,y,z\in L$. We show $(x\diamond y)\diamond z = x\diamond (y\diamond z)$.

If $x\wedge y\neq \hat 0$, then $x\diamond y=0$, so $(x\diamond y)\diamond z=0$.
If we also have $y\wedge z\neq \hat 0$, then $y\diamond z=0$ and
$x\diamond (y\diamond z)=0$. If not, then
$y\diamond z=y\vee z$, and since $y\le y\vee z$ we have
\[
x\wedge (y\vee z)\ge x\wedge y\neq \hat 0,
\]
so again $x\diamond (y\diamond z)=0$. Thus if $x\wedge y\neq \hat 0$, both sides vanish. By symmetry, the same holds if $y\wedge z\neq \hat 0$.

It remains to consider the case $x\wedge y=\hat 0$ and $y\wedge z=\hat 0$, whence $x\diamond y=x\vee y$ and $y\diamond z=y\vee z$.
It follows that
\[
(x\diamond y)\diamond z=
\begin{cases}
x\vee y\vee z,& (x\vee y)\wedge z=\hat 0,\\
0,& (x\vee y)\wedge z\neq \hat 0,
\end{cases}
\]
and
\[
x\diamond (y\diamond z)=
\begin{cases}
x\vee y\vee z,& x\wedge (y\vee z)=\hat 0,\\
0,& x\wedge (y\vee z)\neq \hat 0.
\end{cases}
\]
Thus it suffices to show that $(x\vee y)\wedge z=\hat 0$ if and only if $x\wedge (y\vee z)=\hat 0$.

Since $x\wedge y=\hat 0$ and $y\wedge z=\hat 0$, metric modularity gives
\[
\rk(x\vee y)= \rk(x)+\rk(y),
\qquad
\rk(y\vee z)=\rk(y)+\rk(z).
\]
Suppose first that $x\wedge (y\vee z)=\hat 0$. Then
\[
\rk(x\vee y\vee z)=\rk(x)+\rk(y\vee z) = \rk(x)+\rk(y)+\rk(z).
\]
From these identities it follows that
\[
\rk((x\vee y)\wedge z)
=
\rk(x\vee y)+\rk(z)-\rk(x\vee y\vee z)=0,
\]
hence $(x\vee y)\wedge z=\hat 0$. The other direction follows by symmetry of assumptions. \qedhere
\end{proof}

\begin{remark}
    If the lattice is not modular, then the diamond product is generally not associative. Indeed, consider the lattice of flats of the uniform matroid $U_{3,4}$. For convenience we will write $xy$ for $x\vee y$. Let $a,b,c,d$ be the atoms. We have that $(a\diamond b)\diamond cd = ab\diamond cd = 1$ since $ab\wedge cd = \hat 0$. On the other hand, we have that $a\diamond (b\diamond cd) = a\diamond bcd = a\diamond 1 = 0$.
\end{remark}

\section{Creation Operators and the Hamiltonian}

From a commutative, but not necessarily associative, bilinear operator $\ast$ on a vector space $V$, there is a canonical way to pass to an associative, but not necessarily commutative, algebra structure.

\begin{defn}
For each atom $a\in \At(L)$ define the \emph{creation operator}
\[
    L_a:\R[L]\to \R[L],\qquad L_a(x)=a\diamond x.
\]
\end{defn}

We equip $\R[L]$ with the standard inner product making the lattice basis orthonormal:
\[
\ip{e_x}{e_y}=\de_{x,y}.
\]
Relative to this inner product, $L_a^t$ is the transpose (adjoint) of $L_a$. Explicitly, we have
\[
    L_a^t(y) = \sum_{x:\ x\vee a=y,\ a\wedge x=0} e_x.
\]

\begin{defn}
Define the self-adjoint operator
\[
H_a:=\frac{L_a+L_a^t}{2}.
\]
The \emph{lattice Hamiltonian} is
\[
H:=\sum_{a\in \At(L)} H_a.
\]
The intent of the operator $H$ is to encode ``bulk statistical'' information about the lattice combinatorics into a form which is amenable to spectral and operator-algebraic techniques.
\end{defn}

\begin{remark}
The operators $L_a$ should be viewed as creation operators attached to atoms, while $L_a^t$ are the corresponding annihilation operators. The Hamiltonian $H$ is then the finite-dimensional analogue of a creation--annihilation sum.
\end{remark}

The relevant grading is the rank grading.

\begin{defn}
Let
\[
V_k:=\Span\{e_x:\rk(x)=k\},\qquad 0\le k\le r.
\]
Then
\[
\R[L]=\bigoplus_{k=0}^{r} V_k.
\]
\end{defn}

\begin{lem}
For every atom $a$ and every $k$,
\[
L_a(V_k)\subseteq V_{k+1}.
\]
Equivalently, if $x\in L$ then either $L_a(x)=0$ or $\rk(L_a(x))=\rk(x)+1$.
\end{lem}

\begin{proof}
If $L_a(x)\neq 0$, then by definition $a\wedge x=\zero$ and $L_a(x)=a\vee x$. We claim that $a\vee x$ covers $x$, so that $\rk(a\vee x)=\rk(x)+1$.

Since $a$ is an atom, $a$ covers $\zero=a\wedge x$.  By semimodularity of $L$, whenever $a$ covers $a\wedge x$ it follows that $a\vee x$ covers $x$.
Hence
\[
\rk(a\vee x)=\rk(x)+1.
\]
This use of semimodularity is the only point in the construction at which the geometric-lattice hypothesis is used.
\end{proof}

\begin{cor}
For every atom $a$ and every $k$,
\[
L_a^t(V_k)\subseteq V_{k-1}.
\]
Hence
\[
H(V_k)\subseteq V_{k-1}\oplus V_{k+1}.
\]
\end{cor}

\begin{proof}
    The first statement follows by transposition from the previous lemma; the second is immediate from the definition of $H$.
\end{proof}

\begin{remark}
    Thus the Hamiltonian is rank-bipartite, which immediately implies that all odd vacuum moments $\ip{e_{\zero}}{H^{2m+1}e_{\zero}}$ vanish.
\end{remark}


\section{Radial Compression and Jacobi Coefficients}

The previous corollary says that the Hamiltonian is block tridiagonal with respect to the rank decomposition.  We now compress further to the rank-radial subspace.

\begin{defn}
For $0\le k\le r$ let
\[
L_k:=\{x\in L:\rk(x)=k\},\qquad n_k:=|L_k|.
\]
Define the normalized rank-radial vectors
\[
\rho_k:=\frac{1}{\sqrt{n_k}}\sum_{x\in L_k} e_x.
\]
The \emph{radial subspace} is
\[
\cc R:=\Span\{\rho_0,\rho_1,\dots,\rho_r\}.
\]
\end{defn}

The following result is again apparent by the results of the previous section.

\begin{prop}
Let $P_{\rad}$ denote the orthogonal projection onto $\cc R$.  Then the compression
\[
J:=P_{\rad}HP_{\rad}
\]
is tridiagonal, that is, 
\[
J\rho_k=\beta_{k-1}\rho_{k-1}+\beta_k\rho_{k+1},
\]
for suitable coefficients $\beta_k\ge 0$, with the conventions $\beta_{-1}=\beta_r=0$.
Moreover the diagonal entries vanish:
\[
\ip{\rho_k}{J\rho_k}=0.
\]
\end{prop}

We now derive the basic explicit formula.

\begin{defn}
For $x\in L$ let
\[
a(x):=|\{p\in \At(L):p\le x\}|
\]
denote the number of atoms below $x$.
\end{defn}

\begin{lem}
If $x\lessdot y$ is a cover, then the number of atoms $p$ such that
\[
x\diamond p=y
\]
is exactly
\[
a(y)-a(x).
\]
\end{lem}

\begin{proof}
    By definition, $x\diamond p=y$ means $p\wedge x=\zero$ and $x\vee p=y$.  Since $p$ is an atom, this is equivalent to $p\le y$ but $p\nleq x$.  Therefore the number of such atoms is precisely the number of atoms below $y$ that are not below $x$, namely $a(y)-a(x)$.
\end{proof}

\begin{thm}\label{thm:beta-formula}
The Jacobi coefficients of the radial compression are
\[
\beta_k=
\frac{1}{2\sqrt{n_k n_{k+1}}}
\sum_{\rk(x)=k,\ x\lessdot y}
\bigl(a(y)-a(x)\bigr),
\qquad 0\le k\le r-1.
\]
\end{thm}

\begin{proof}
Let
\[
A:=\sum_{a\in \At(L)} L_a.
\]
Then
\[
H=\frac{A+A^t}{2}.
\]
Since $A$ raises rank by one, we have
\[
\beta_k=\ip{\rho_k}{H\rho_{k+1}}
=\frac12 \ip{\rho_k}{A^t\rho_{k+1}}
=\frac12 \ip{A\rho_k}{\rho_{k+1}}.
\]

Now
\[
A\rho_k=
\frac{1}{\sqrt{n_k}}\sum_{\rk(x)=k}\sum_{a\in \At(L)} e_{a\diamond x}.
\]
The coefficient of a fixed $y$ with $\rk(y) =k+1$ in this sum is exactly the number of pairs $(x,a)$ with
$\rk(x)=k$ and $x\diamond a=y$, which by the lemma equals
\[
\sum_{\rk(x)=k,\ x\lessdot y} (a(y)-a(x)).
\]
Therefore
\[
\ip{A\rho_k}{\rho_{k+1}}
=
\frac{1}{\sqrt{n_k n_{k+1}}}
\sum_{\rk(x)=k,\ x\lessdot y}
(a(y)-a(x)).
\]
Multiplying by $1/2$ gives the formula.
\end{proof}

We record the general determinant recurrence and its resolvent consequence.

\begin{prop}\label{prop:det-rec}
Let
\[
J=
\begin{pmatrix}
0 & \beta_0 \\
\beta_0 & 0 & \beta_1 \\
& \beta_1 & 0 & \ddots \\
&& \ddots & \ddots & \beta_{r-1}\\
&&& \beta_{r-1} & 0
\end{pmatrix}.
\]
Let $J_k$ be the upper-left $(k+1)\times (k+1)$ principal minor, and set
\[
D_k(t):=\det(I-tJ_k),\qquad D_{-1}(t):=1.
\]
Then
\[
D_0(t)=1
\]
and
\[
D_{k+1}(t)=D_k(t)-\beta_k^2 t^2 D_{k-1}(t)
\qquad (0\le k\le r-1).
\]
\end{prop}

\begin{proof}
This follows from the standard continuant recurrence for Jacobi matrices and can be proven inductively by Laplace expansion. See, for example, \cite[Exercise I.5.7]{Chihara} or \cite{ReedSimon-IV}.
\end{proof}

\begin{cor}\label{prop:resolvent}
Let
\[
G(t):=\ip{\rho_0}{(I-tJ)^{-1}\rho_0} = \sum_{k=0}^\infty t^k\ip{\rho_0}{J^k\rho_0}.
\]
Then
\[
G(t)=\frac{D_{r-1}(t)}{D_r(t)}.
\]
\end{cor}

\begin{proof}
This follows from the usual Schur-complement formula for Jacobi matrices; see \cite{ReedSimon-IV}.
\end{proof}

\begin{remark}
    In general the radial subspace need not be invariant under $H$; however, in highly symmetric examples such as Boolean lattices and projective geometries, the radial subspace is invariant and $J$ is the actual radial restriction. Indeed, in both cases the automorphism group acts transitively on each rank layer, and $H$ commutes with the natural action of the automorphism group.
\end{remark}

\begin{lem}\label{lem:vacuum-moment-compression}
Assume that the radial subspace $\cc R$ is invariant under $H$, and identify $\rho_0$ with the
first standard basis vector $e_0$ of the Jacobi model $J=H|_{\cc R}$.  Then for every $k\ge 0$,
\[
\ip{e_{\zero}}{H^k e_{\zero}}=\ip{e_{\zero}}{J^k e_{\zero}}.
\]
More generally, for every polynomial $p$,
\[
\ip{e_{\zero}}{p(H)e_{\zero}}=\ip{e_{\zero}}{p(J)e_{\zero}}.
\]
\end{lem}

\begin{proof}
Since $e_{\zero}=\rho_0\in \cc R$ and $\cc R$ is $H$-invariant, one has
\[
H^k\rho_0=J^k\rho_0
\]
for every $k\ge 0$, where the right-hand side is computed in the radial basis
$\{\rho_0,\dots,\rho_r\}$.  Taking the inner product with $\rho_0$ gives
\[
\ip{e_{\zero}}{H^k e_{\zero}}
=
\ip{\rho_0}{H^k\rho_0}
=
\ip{\rho_0}{J^k\rho_0}
=
\ip{e_0}{J^k e_0}.
\]
The polynomial statement follows by linearity.
\end{proof}

\begin{remark}
From the operator-theoretic perspective, the more natural definition for $G(t)$ is the vacuum resolvent
\[
G(t)=\ip{e_{\zero}}{(I-tH)^{-1}e_{\zero}}
\]
whenever the radial subspace is cyclic for the compressed dynamics.  However, for the main examples computed in this note, the radial model captures the vacuum moments completely.
\end{remark}

\section{Examples}

\subsection{The modular lattice $M_3$}

We now work out the smallest non-Boolean geometric lattice in detail.
Let
\[
M_3=\{\zero,x,y,z,\one\}
\]
with three atoms $x,y,z$ and pairwise joins equal to the top.  The rank decomposition is
\[
L_0=\{\zero\},\qquad L_1=\{x,y,z\},\qquad L_2=\{\one\}.
\]
Thus
\[
n_0=1,\qquad n_1=3,\qquad n_2=1.
\]
Also
\[
a(\zero)=0,\qquad a(x)=a(y)=a(z)=1,\qquad a(\one)=3.
\]

Applying Theorem \ref{thm:beta-formula} gives
\[
\beta_0=\frac{1}{2\sqrt{1\cdot 3}}\cdot 3=\frac{\sqrt{3}}{2},
\]
since there are three covers $\zero\lessdot x,\zero\lessdot y,\zero\lessdot z$, each contributing $1$.
For the top coefficient, each cover $x\lessdot \one$, $y\lessdot \one$, $z\lessdot \one$ contributes
\[
a(\one)-a(x)=2,
\]
hence
\[
\beta_1=
\frac{1}{2\sqrt{3\cdot 1}} \cdot 6
=\sqrt{3}.
\]

\begin{prop}
For $M_3$, the radial Jacobi matrix is
\[
J_{M_3}=
\begin{pmatrix}
0 & \frac{\sqrt{3}}{2} & 0\\
\frac{\sqrt{3}}{2} & 0 & \sqrt{3}\\
0 & \sqrt{3} & 0
\end{pmatrix}.
\]
\end{prop}

\begin{proof}
This is immediate from the computed coefficients.
\end{proof}

Applying the Jacobi recurrence formula to $M_3$ gives
\[
G_{M_3}(t):=\ip{\rho_0}{(I-tJ_{M_3})^{-1}\rho_0}
=\frac{D_1(t)}{D_2(t)}
=
\frac{1-\frac34 t^2}{1-\frac{15}{4}t^2}.
\]

\subsection{Boolean lattices and the Krawtchouk Jacobi matrix}

Let $\BB_n$ denote the Boolean lattice on $n$ atoms.

\begin{prop}\label{prop:boolean-beta}
For the Boolean lattice $\BB_n$,
\[
\beta_k=\frac12 \sqrt{(k+1)(n-k)},
\qquad 0\le k\le n-1.
\]
\end{prop}

\begin{proof}
A rank-$k$ element is simply a $k$-subset of $[n]$, so
\[
n_k=\binom{n}{k}.
\]
Moreover
\[
a(x)=\rk(x)=k
\]
for every rank-$k$ element.  If $x\lessdot y$, then $|y|=|x|+1$, hence
\[
a(y)-a(x)=1.
\]
The number of covers from rank $k$ to rank $k+1$ is
\[
\binom{n}{k}(n-k)=\binom{n}{k+1}(k+1).
\]
Therefore
\[
\beta_k=
\frac{1}{2\sqrt{\binom{n}{k}\binom{n}{k+1}}}\binom{n}{k}(n-k)
=\frac12\sqrt{(k+1)(n-k)}.
\]
\end{proof}

\begin{remark}
These coefficients agree with the off-diagonal Jacobi parameters of the symmetric Krawtchouk polynomials with $p=\frac{1}{2}$: see \cite{Chihara} or \cite[Chapter 9]{Koekoek}. More precisely, the standard Krawtchouk Jacobi matrix, corresponding to multiplication by the coordinate $x$, has the same off-diagonal coefficients and constant diagonal $n/2$. Thus the Boolean-lattice Hamiltonian recovers the \emph{centered} Krawtchouk Jacobi matrix. In particular, the associated orthogonal polynomials are the Krawtchouk polynomials up to an affine change of variable.
\end{remark}

For later reference we record the first few determinant polynomials.

\begin{example}
For $\BB_1$ one gets
\[
J=
\begin{pmatrix}
0 & \frac12\\
\frac12 & 0
\end{pmatrix},
\qquad
\det(I-tJ)=1-\frac14 t^2.
\]
For $\BB_2$ one gets
\[
\beta_0=\beta_1=\frac{\sqrt{2}}{2},
\qquad
\det(I-tJ)=1-t^2.
\]
For $\BB_3$ one gets
\[
\beta_0=\frac{\sqrt{3}}{2},\qquad \beta_1=1,\qquad \beta_2=\frac{\sqrt{3}}{2},
\]
hence
\[
\det(I-tJ)=1-\frac52 t^2+\frac{9}{16}t^4.
\]
\end{example}

Equivalently, if $D^{(n)}_k(t)$ denotes the determinant polynomial for the leading principal
$(k+1)\times (k+1)$ minor, then
\[
D^{(n)}_{-1}(t)=1,\qquad D^{(n)}_0(t)=1,
\]
and
\[
D^{(n)}_{k+1}(t)=D^{(n)}_k(t)-\frac{(k+1)(n-k)}{4}t^2 D^{(n)}_{k-1}(t),
\qquad 0\le k\le n-1,
\]
with
\[
G_{\BB_n}(t)=\frac{D^{(n)}_{n-1}(t)}{D^{(n)}_n(t)}.
\]

There is also a closed spectral-sum formula.  Since $J_{\BB_n}$ is the radial compression of one half of the hypercube adjacency operator, its eigenvalues are
\[
\lambda_j=\frac{n}{2}-j,
\qquad 0\le j\le n,
\]
and the vacuum spectral weights are the binomial weights $2^{-n}\binom{n}{j}$.  Thus
\[
G_{\BB_n}(t)
=
\sum_{j=0}^{n}
\frac{2^{-n}\binom{n}{j}}{1-\left(\frac{n}{2}-j\right)t}.
\]

\subsection{Projective geometries and $q$-Hahn-type coefficients}

Let $L=\PG(r-1,q)$ denote the lattice of linear subspaces of $\bb F_q^r$, ordered by inclusion.

\begin{prop}
For $\PG(r-1,q)$, the radial Jacobi coefficients are
\[
\beta_k=\frac{q^k}{2}\sqrt{[k+1]_q\,[r-k]_q},
\qquad 0\le k\le r-1,
\]
where
\[
[m]_q:=\frac{q^m-1}{q-1}.
\]
\end{prop}

\begin{proof}
We refer to \cite[Section 1.4]{Dembowski} or \cite{Hirschfeld} for standard facts about the dimensions of the rank subsets of projective geometries. The number of $k$-dimensional subspaces of $\mathbb{F}_q^r$ is given by the Gaussian binomial coefficient,
\[
N_k(r,q)=\binom{r}{k}_q,
\]
and in particular the number of points (i.e.\ $1$-dimensional subspaces) is
\[
N_1(r,q)=\frac{q^r-1}{q-1} = [r]_q.
\]

We compute the Jacobi coefficients associated to the Hamiltonian in the radial basis indexed by rank. Let $x$ be a $k$-dimensional subspace.

\medskip

\noindent\emph{Upward neighbors.}
The number of $(k+1)$-dimensional subspaces containing $x$ is
\[
a_k = [r-k]_q,
\]
since such subspaces correspond to $1$-dimensional subspaces of the quotient $\mathbb F_q^r/x$.

\medskip

\noindent
\emph{Downward neighbors.}
The number of $k$-dimensional subspaces contained in a fixed $(k+1)$-dimensional subspace $y$ is
\[
b_{k+1} = [k+1]_q.
\]

\medskip

\noindent
\emph{Atom multiplicity.}
For a cover $x<y$, the number of atoms $a$ such that $a\diamond x=y$ is the number of points of $y$ not contained in $x$. Using the point count above,
\[
|y\setminus x|
=
N_1(k+1,q)-N_1(k,q)
=
q^k.
\]

\medskip

Combining these, the radial coefficient is
\[
\beta_k
=
\frac{q^k}{2}\sqrt{[k+1]_q\,[r-k]_q}. \qedhere
\]
\end{proof}

\begin{remark}
This is the $q$-analogue of Proposition \ref{prop:boolean-beta}.  In other words, projective geometries produce Jacobi matrices which are consistent with the recurrence coefficients of $q$-Hahn-type orthogonal polynomials \cite[Chapter 14]{Koekoek}. We do not attempt to identify the precise normalization here.
\end{remark}

For $L=\PG(r-1,q)$ the radial subspace is invariant under the Hamiltonian, so the vacuum resolvent of $H$ agrees with the Jacobi resolvent of the compressed matrix.
Therefore, if $D^{(r,q)}_k(t)$ denotes the determinant polynomial of the leading principal
minor of size $(k+1)\times (k+1)$, then we have
\[
D^{(r,q)}_{-1}(t)=1,\qquad D^{(r,q)}_0(t)=1,
\]
and
\[
D^{(r,q)}_{k+1}(t)
=
D^{(r,q)}_k(t)-\frac{q^{2k}}{4}[k+1]_q[r-k]_q\,t^2\,D^{(r,q)}_{k-1}(t),
\qquad 0\le k\le r-1,
\]
with
\[
G_{\PG(r-1,q)}(t)=\frac{D^{(r,q)}_{r-1}(t)}{D^{(r,q)}_{r}(t)}.
\]


\begin{remark}
    Let $L=\AG(r,q)$ denote the lattice of affine subspaces of $\mathbb{F}_q^r$, ordered by inclusion, with a bottom element adjoined. We recall that affine $k$-flats are cosets of $k$-dimensional linear subspaces. The number of such flats is
    \[
    M_k(r,q)=q^{r-k}N_k(r,q),
    \]
    where $N_k(r,q)$ is the Gaussian binomial coefficient; see again \cite[Section 1.4]{Dembowski} or \cite{Hirschfeld}.

    Although this lattice is not geometric in the matroidal sense, it is finite, atomic, and semimodular, so the preceding construction applies. A computation analogous to the projective case leads to the radial coefficients
    \[
    \beta_k
    =
    \frac{(q-1)q^{k-1}}{2}\sqrt{q\,[k]_q\,[r-k+1]_q}
    =
    \frac12(q-1)q^{k-\frac12}\sqrt{[k]_q\,[r-k+1]_q}.
    \]
    These coefficients exhibit features reminiscent of the $q$-Hahn family, although we do not pursue a precise identification here.
\end{remark}


\section{Products of Lattices}

A key structural property of the Hamiltonian is its behavior under products of geometric lattices. 

\begin{prop}\label{prop:product}
Let $L = L' \times L''$ be a product of finite geometric lattices.
Then there is a canonical identification
\[
\R[L] \cong \R[L'] \otimes \R[L'']
\]
under which the Hamiltonian decomposes as a Kronecker sum
\[
H_L = H_{L'} \otimes I + I \otimes H_{L''}.
\]
\end{prop}

\begin{proof}
The atoms of $L' \times L''$ are precisely the elements
$(a, \zero'')$ for $a \in \At(L')$ and $(\zero', b)$ for
$b \in \At(L'')$.  Under the identification
\[
e_{(x',x'')} \leftrightarrow e_{x'} \otimes e_{x''},
\]
we verify from the definition of the diamond product that for
$x = (x', x'') \in L$,
\[
(a,\zero'') \diamond (x',x'')
=
\begin{cases}
(a \vee x',\, x''), & a \wedge x' = \zero',\\
0, & a \wedge x' \neq \zero'.
\end{cases}
\]
Hence $L_{(a,\zero'')} = L_a \otimes I$ and similarly
$L_{(\zero',b)} = I \otimes L_b$.  Taking self-adjoint parts gives
\[
H_{(a,\zero'')} = H_a \otimes I,
\qquad
H_{(\zero',b)} = I \otimes H_b.
\]
Summing over all atoms of $L$ yields
\[
H_L = H_{L'} \otimes I + I \otimes H_{L''}.
\qedhere
\]
\end{proof}

\begin{cor}[Shuffle formula]\label{cor:shuffle}
Let $L = L' \times L''$ and let $x = (x', x'')$, $y = (y', y'')$ with
$x \leq y$.  Set
\[
d_1 = \rk(y') - \rk(x'),
\qquad
d_2 = \rk(y'') - \rk(x''),
\qquad
d = d_1 + d_2.
\]
Then
\[
\ip{e_x}{H_L^d e_y}
=
\binom{d}{d_1}
\ip{e_{x'}}{H_{L'}^{d_1} e_{y'}}
\ip{e_{x''}}{H_{L''}^{d_2} e_{y''}}.
\]
\end{cor}

\begin{proof}
From the binomial expansion of $H_L^d = (H_{L'}\otimes I + I\otimes H_{L''})^d$,
\[
H_L^d = \sum_{k=0}^d \binom{d}{k}
(H_{L'})^k \otimes (H_{L''})^{d-k}.
\]
Since $H_L$ changes rank by exactly one at each step, the coefficient
$\ip{e_x}{H_L^d e_y}$ is nonzero only when $d = \rk(y) - \rk(x)$,
which holds by assumption.  Any word of length $d$ that
moves from $x$ to $y$ must raise rank at every step; the number
of steps assigned to the $L'$-factor must equal $d_1$ and to the
$L''$-factor $d_2$.  Hence only the $k = d_1$ term contributes.
\end{proof}

\begin{remark}
    The binomial coefficient $\binom{d}{d_1}$ counts the interleavings of $d_1$ rank-raising steps in $L'$ with $d_2$ rank-raising steps in $L''$. Thus $\ip{e_x}{H_L^d e_y}$ counts ordered minimal atomic decompositions of the interval $[x,y]$ in $L$, weighted by $2^{-d}$, and the product formula reflects the fact that every such decomposition is a shuffle of a decomposition in $[x',y']$ with one in $[x'',y'']$.
\end{remark}

Recall that the \emph{vacuum spectral measure} of a self-adjoint operator $T$ on $\R[L]$ is the compactly supported probability measure $\mu$ determined by
\[
\int t^k \, d\mu(t) = \ip{e_{\zero}}{T^k e_{\zero}}, \qquad k \ge 0.
\]

\begin{prop}\label{thm:conv}
Let $L = L' \times L''$ be a product of finite geometric lattices, and let $\mu_L$, $\mu_{L'}$, $\mu_{L''}$ denote the vacuum spectral measures of $H_L$, $H_{L'}$, $H_{L''}$ respectively.  Then
\[
\mu_L = \mu_{L'} \ast \mu_{L''},
\]
where $\ast$ denotes the convolution of measures.
\end{prop}

\begin{proof}
The state $\varphi = \ip{e_{\zero'} \otimes e_{\zero''}}{\,\cdot\, e_{\zero'} \otimes e_{\zero''}}$
is a product vector state.  The state $\varphi$ is a product state, and the operators $H_{L'} \otimes I$ and $I \otimes H_{L''}$ commute and act on separate tensor factors. It follows that the moments factor as
\[
\varphi\!\left((H_{L'} \otimes I + I \otimes H_{L''})^k\right)
=
\sum_{j=0}^k \binom{k}{j}
\varphi'\!\bigl((H_{L'})^{j}\bigr)
\varphi''\!\bigl((H_{L''})^{k-j}\bigr),
\]
which is precisely the moment formula for the classical convolution $\mu_{L'}\ast \mu_{L''}$.
\end{proof}

\begin{remark}
    For the Boolean lattice $B_n \cong B_1^{\times n}$, the Hamiltonian is an $n$-fold tensor sum of a fixed $2 \times 2$ matrix. Hence the vacuum spectral measure is the $n$-fold classical convolution of the rank-one case, recovering the Krawtchouk family from a single $2\times 2$ model.
\end{remark}

\begin{remark}
    The product formula of Proposition \ref{thm:conv} shows that the vacuum spectral measure behaves functorially with respect to lattice products. It is natural to ask whether there is a corresponding local structure at the level of intervals. One might expect that suitable contributions from intervals $[a,b]$ play a role analogous to cumulants in the sense of Speicher’s combinatorial approach to free probability \cite{Speicher}. We do not pursue this here, but it would be interesting to make this connection precise.
\end{remark}

\section{Concluding Remarks}

The construction developed here is elementary but captures a surprising vantage on the combinatorial structure of the lattice,  providing a canonical and computable bridge between finite geometric lattices and finite Jacobi matrices. There are several directions of further inquiry, including a development of a commutator calculus for the operators $L_a$ for which richer uses of ideas from noncommutative analysis may have purchase. We end with several natural questions.

\begin{question}
For which geometric lattices is the radial subspace actually invariant under the Hamiltonian?
\end{question}

\begin{question}
Which lattice families produce classical hypergeometric Jacobi matrices?
\end{question}

\begin{question}
Can the full operator $H$, beyond its radial compression, be used to extract further lattice invariants or refined combinatorial data?
\end{question}

\section*{Acknowledgments}
The author thanks Jos\'e Contreras Mantilla and Zhiyuan Yang for useful comments. 
The author used OpenAI’s ChatGPT for editorial assistance and for exploring examples and formulations during the development of this work. All mathematical content, results, and conclusions are solely those of the author.

\end{document}